\documentclass[11pt]{amsart}

\usepackage[T1]{fontenc}
\usepackage[utf8]{inputenc}
\usepackage{lmodern}
\usepackage{amsmath,amssymb,amsthm,mathtools}
\usepackage{enumitem}
\usepackage[
 pdfauthor={},
 pdftitle={},
 pdfcreator={},
 colorlinks=true,
 linkcolor=blue,
 citecolor=blue,
 urlcolor=blue
]{hyperref}

\theoremstyle{plain}
\newtheorem{theorem}{Theorem}[section]
\newtheorem{proposition}[theorem]{Proposition}
\newtheorem{corollary}[theorem]{Corollary}
\newtheorem{lemma}[theorem]{Lemma}

\theoremstyle{definition}

\newtheorem{example}[theorem]{Example}

\theoremstyle{remark}
\newtheorem{remark}[theorem]{Remark}

\numberwithin{equation}{section}

\title[Generalized divisor topology of commutative rings]
{Generalized divisor topology of commutative rings}

\author[Koç]{Suat Koç$^{*}$}
\address{Department of Mathematics, Marmara University, İstanbul, Türkiye}
\email{suat.koc@marmara.edu.tr}
\thanks{$^{*}$Corresponding author.}
\author[Doğan]{İrem Doğan}
\address{Department of Mathematics, Marmara University, İstanbul, Türkiye}
\email{iremsarikoc@gmail.com}
\author[Erdemir]{Dilara Erdemir}
\address{Department of Mathematics, Yildiz Technical University, İstanbul, Türkiye}
\email{dilaraer@yildiz.edu.tr}
\author[Tekir]{Ünsal Tekir}
\address{Department of Mathematics, Marmara University, İstanbul, Türkiye}
\email{utekir@marmara.edu.tr}

\subjclass[2020]{Primary 13A15; Secondary 13A05, 54H10, 54H12}
\keywords{divisor topology, generalized divisor topology,
	radicals of principal ideals, Alexandrov space, Kolmogorov quotient,
	von Neumann regular ring, G-domain}

\begin{document}

\begin{abstract}
	Let $R$ be a commutative ring with nonzero identity and let $R^\#$
	denote the set of its nonzero nonunits. We extend the divisor topology
	$D(R)$, previously studied for integral domains, to arbitrary commutative
	rings and introduce the generalized divisor topology $GD(R)$ on
	$EC(R^\#)$. Its basic open sets are
	\[
	B_a=\{[b]\in EC(R^\#): b\mid a^n
	\text{ for some }n\geq 1\}.
	\]
	The relation
	\[
	[b]\in B_a
	\quad\Longleftrightarrow\quad
	\sqrt{aR}\subseteq\sqrt{bR}
	\]
	shows that $GD(R)$ records radical divisibility among principal ideals.
	We prove that $GD(R)$ is an Alexandrov space and identify its Kolmogorov
	quotient with the poset of radicals of nonzero proper principal ideals.
	This description yields characterizations of the $T_0$ and discrete
	properties and of the equality $GD(R)=D(R)$. We also determine the
	isolated points of $GD(R)$. Further, we characterize nestedness,
	compactness, the Lindel\"of property, and Noetherianity in terms of the
	order structure of radicals of principal ideals. In particular, for an
	integral domain $R$, $GD(R)$ is compact if and only if $R$ is a
	$G$-domain, while for a UFD the Lindel\"of and Noetherian properties are
	determined by the number of nonassociate prime elements. Finally, we study the interaction of $GD(R)$ with multiplication and
	describe the behavior of its Kolmogorov quotient under surjective
	homomorphisms with nil kernel.
\end{abstract}

\maketitle

\section{Introduction and basic reductions}

Divisibility provides a natural link between the multiplicative structure of a
ring and order-theoretic topology. In earlier work \cite{YigitKoc}, Yi\u{g}it and Ko\c{c}
introduced the divisor topology $D(R)$ for an integral domain $R$. If $R^\#$
denotes the set of nonzero nonunits of $R$ and elements are identified when
they generate the same principal ideal, then the basic open set determined by
$a\in R^\#$ is
\[
U_a=\{[b]\in EC(R^\#): b\mid a\}.
\]
The resulting topology reflects several familiar properties of domains. It is
an Alexandrov $T_0$-space, its isolated points correspond to irreducible
elements, and nestedness characterizes valuation domains. Further results concern countability, Noetherianity, and the infinitude
of prime elements in a UFD; the divisor-topological construction was
later extended to modules over domains in
\cite{divisortopologyofmodules}.

The domain hypothesis plays a substantial role in this picture. Once zero
divisors are allowed, ordinary divisibility still defines a topology, but it
does not record some of the ideal-theoretic information carried by powers and
radicals. The purpose of the present paper is therefore twofold. First, we
consider the divisor topology $D(R)$ for an arbitrary commutative ring.
Second, we introduce a radical version of this topology that is better adapted
to the presence of zero divisors and nilpotent elements.

Let $R$ be a commutative ring with nonzero identity and let
\[
R^\#=R\setminus\bigl(U(R)\cup\{0\}\bigr),
\]
where $U(R)$ denotes the group of units of $R$. For $a,b\in R^\#$, write
$a\sim b$ if $aR=bR$, and let $EC(R^\#)$ denote the corresponding set of
equivalence classes. For $a\in R^\#$, put
\[
B_a=\{[b]\in EC(R^\#):b\mid a^n
\text{ for some }n\geq1\}.
\]
The family $\{B_a:a\in R^\#\}$ forms a basis for a topology on
$EC(R^\#)$, which we call the \emph{generalized divisor topology} and
denote by $GD(R)$.

The passage from $D(R)$ to $GD(R)$ has a direct algebraic interpretation.
For $a,b\in R^\#$,
\[
[b]\in B_a
\quad\Longleftrightarrow\quad
b\mid a^n\text{ for some }n\geq1
\quad\Longleftrightarrow\quad
\sqrt{aR}\subseteq\sqrt{bR}.
\]
Thus $D(R)$ records ordinary divisibility among principal ideals, whereas
$GD(R)$ records divisibility after passing to their radicals. For a
nonnilpotent element $a$,
\[
B_a=\bigcup_{n\geq1}U_{a^n},
\]
while every nonzero nilpotent element satisfies
\[
B_a=EC(R^\#).
\]
These two observations account for much of the difference between the divisor
and generalized divisor topologies.

The radical order appearing here is related to structures already considered in
commutative algebra. In particular, radicals of principal ideals have
been studied from an order-theoretic viewpoint; see, for example,
\cite{Spirito}. Spaces of radical ideals endowed with the hull-kernel
topology have also been investigated in
\cite{FinocchiaroFontanaSpirito}. Our construction is different from
the latter: rather than placing a topology directly on a space of
radical ideals, the generalized divisor topology is defined on
$EC(R^\#)$ through divisibility by powers.

The radical description above is the main reduction used throughout the paper.
We show that $GD(R)$ is an Alexandrov space and that two points are
topologically indistinguishable precisely when the corresponding principal
ideals have the same radical. Consequently, the Kolmogorov quotient of
$GD(R)$ is naturally identified with the poset of radicals of principal
ideals arising from elements of $R^\#$, endowed with its upper Alexandrov
topology. This identification allows topological properties of $GD(R)$ to be
translated into order-theoretic properties of these radicals.

The first main results concern the information lost in passing from ordinary
divisibility to radical divisibility. We characterize the rings for which $GD(R)=D(R)$ and those for which
$GD(R)$ is $T_0$ or discrete
(Theorems~\ref{thm:T0-classification} and~\ref{thm:discrete}). We also determine the isolated points of $GD(R)$
(Proposition~\ref{prop:isolated}) and study the finite spaces arising
from $\mathbb Z_n$. A second group of results concerns global properties. We show that $GD(R)$ is nested if and only if the prime ideals of $R$
are linearly ordered, equivalently if the Zariski topology on
$\operatorname{Spec}(R)$ is nested
(Theorem~\ref{thm:nested}). Compactness and the Lindel\"of property are described in terms of
finite and countable coinitial families of radicals of principal
ideals, respectively, and Noetherianity is characterized by a
well-quasi-order condition. In particular, for an integral domain $R$, the space $GD(R)$ is
compact if and only if $R$ is a $G$-domain
(Theorem~\ref{thm:compact-G-domain}). For a UFD, the Lindel\"of and Noetherian properties are
determined by the number of nonassociate prime elements. Finally, we study the multiplicative structure of $GD(R)$ and the
behavior of its Kolmogorov quotient under surjective homomorphisms
with nil kernel.

Section~2 studies separation properties and the equality $GD(R)=D(R)$.
Section~3 treats isolated points and arithmetic examples. Section~4 is devoted
to nestedness, compactness, the Lindel\"of property, and Noetherianity.
Section~5 considers multiplicative and quotient constructions.

\subsection{Divisor and radical divisibility}

For completeness, we first record the elementary properties that will be used
throughout the paper. The definition of $B_a$ is independent of the choice of
representatives, since equality of principal ideals is preserved under taking
positive powers.

\begin{lemma}\label{lem:radical-divisibility}
	Let $R$ be a ring and let $a,b\in R^\#$. Then
	\[
	[b]\in B_a
	\quad\Longleftrightarrow\quad
	\sqrt{aR}\subseteq\sqrt{bR}.
	\]
	Equivalently,
	\[
	B_a\subseteq B_b
	\quad\Longleftrightarrow\quad
	\sqrt{bR}\subseteq\sqrt{aR}.
	\]
\end{lemma}

\begin{proof}
	We have
	\[
	[b]\in B_a
	\Longleftrightarrow
	b\mid a^n \text{ for some }n\geq1
	\Longleftrightarrow
	a^n\in bR \text{ for some }n\geq1.
	\]
	The last condition is equivalent to $a\in\sqrt{bR}$, and hence to
	$\sqrt{aR}\subseteq\sqrt{bR}$.
	
	For the second assertion, if $B_a\subseteq B_b$, then
	$[a]\in B_a\subseteq B_b$, so by the first assertion,
	\[
	\sqrt{bR}\subseteq\sqrt{aR}.
	\]
	Conversely, suppose that
	$\sqrt{bR}\subseteq\sqrt{aR}$ and let $[x]\in B_a$. Then
	\[
	\sqrt{bR}\subseteq\sqrt{aR}\subseteq\sqrt{xR},
	\]
	and hence $[x]\in B_b$. Therefore, $B_a\subseteq B_b$.
\end{proof}

\begin{proposition}\label{prop:basis}
	Let $R$ be a ring. Then the family
	\[
	\mathfrak B=\{B_a:a\in R^\#\}
	\]
	is a basis for a topology on $EC(R^\#)$. Moreover, $B_a$ is the smallest
	open neighbourhood of $[a]$ for every $a\in R^\#$. Consequently, $GD(R)$
	is an Alexandrov space.
\end{proposition}

\begin{proof}
	Since $[a]\in B_a$ for every $a\in R^\#$,
	\[
	EC(R^\#)=\bigcup_{a\in R^\#}B_a.
	\]
	Suppose that $[c]\in B_a\cap B_b$. By
	Lemma~\ref{lem:radical-divisibility},
	\[
	\sqrt{aR}\subseteq\sqrt{cR}
	\qquad\text{and}\qquad
	\sqrt{bR}\subseteq\sqrt{cR}.
	\]
	If $[x]\in B_c$, then
	\[
	\sqrt{cR}\subseteq\sqrt{xR},
	\]
	and therefore $[x]\in B_a\cap B_b$. Thus
	\[
	[c]\in B_c\subseteq B_a\cap B_b,
	\]
	which proves that $\mathfrak B$ is a basis.
	
	Let $O$ be an open set containing $[a]$. Choose $b\in R^\#$ such that
	\[
	[a]\in B_b\subseteq O.
	\]
	Then
	\[
	\sqrt{bR}\subseteq\sqrt{aR},
	\]
	and Lemma~\ref{lem:radical-divisibility} gives $B_a\subseteq B_b$.
	Hence
	\[
	B_a\subseteq O,
	\]
	so $B_a$ is the smallest open neighbourhood of $[a]$. Therefore every
	point has a smallest open neighbourhood, and $GD(R)$ is Alexandrov.
\end{proof}

The ordinary divisor topology can be defined on an arbitrary commutative ring
by using the basis
\[
U_a=\{[b]\in EC(R^\#):b\mid a\},
\qquad a\in R^\#.
\]
The next proposition records the relation between the two topologies.

\begin{proposition}\label{prop:D-GD}
	Let $R$ be a ring and $a\in R^\#$.
	\begin{enumerate}
		\item If $a$ is not nilpotent, then
		\[
		B_a=\bigcup_{n\geq1}U_{a^n}.
		\]
		\item If $a$ is nilpotent, then
		\[
		B_a=EC(R^\#).
		\]
		\item Every $GD(R)$-open set is $D(R)$-open. Hence $GD(R)$ is coarser
		than $D(R)$.
	\end{enumerate}
\end{proposition}

\begin{proof}
	Assume first that $a$ is not nilpotent. Then $a^n\in R^\#$ for every
	$n\geq1$, and
	\[
	[x]\in B_a
	\Longleftrightarrow
	x\mid a^n \text{ for some }n\geq1
	\Longleftrightarrow
	[x]\in U_{a^n}\text{ for some }n\geq1.
	\]
	This proves (1).
	
	If $a$ is nilpotent, then, since $a\neq 0$, there exists $n\geq 2$ such that $a^n=0$. Every
	$b\in R^\#$ divides $0$, and hence $[b]\in B_a$. This proves (2).
	
	Finally, by (1) every basic open set $B_a$ with $a$ nonnilpotent is a
	union of $D(R)$-open sets, while by (2) a basic open set associated to a
	nilpotent element is the whole space. Thus every basic open set of $GD(R)$
	is open in $D(R)$.
\end{proof}

\subsection{The radical order and the Kolmogorov quotient}

Let $(P,\leq)$ be a partially ordered set. A subset $U$ of $P$ is called
an \emph{upper set} if $x\in U$ and $x\leq y$ imply $y\in U$. The family
of all upper sets forms an Alexandrov topology on $P$; with this convention,
the smallest open neighbourhood of $x\in P$ is the principal upper set
\[
\uparrow x=\{y\in P:x\leq y\}.
\]
We refer to this topology as the \emph{upper Alexandrov topology};
see \cite{Curry} and, for general background on Alexandrov spaces,
\cite{Alex,Arenas}.

We apply this construction to the partially ordered set
\[
\mathcal R_{\mathrm{pr}}(R)
=
\{\sqrt{aR}:a\in R^\#\},
\]
ordered by inclusion. Thus, for $P\in\mathcal R_{\mathrm{pr}}(R)$, its
smallest open neighbourhood is
\[
\uparrow P
=
\{Q\in\mathcal R_{\mathrm{pr}}(R):P\subseteq Q\}.
\]

Recall that two points $x$ and $y$ of a topological space $X$ are
\emph{topologically indistinguishable}, written $x\equiv y$, if they have
the same open neighbourhoods. The quotient of $X$ by this equivalence
relation, endowed with the quotient topology, is called the
\emph{Kolmogorov quotient} (or $T_0$-identification) of $X$ and will be
denoted by $KQ(X)$; see \cite[Section 13C]{Willard}. The space $KQ(X)$
is always $T_0$.

\begin{theorem}\label{thm:KQ}
	Let $R$ be a ring and define
	\[
	\rho:EC(R^\#)\longrightarrow\mathcal R_{\mathrm{pr}}(R),
	\qquad
	\rho([a])=\sqrt{aR}.
	\]
	Then the following statements hold.
	\begin{enumerate}
		\item The map $\rho$ is well defined and surjective, and
		\[
		B_a=\rho^{-1}\bigl(\uparrow\sqrt{aR}\bigr)
		\]
		for every $a\in R^\#$.
		\item For $a,b\in R^\#$, the points $[a]$ and $[b]$ are topologically
		indistinguishable in $GD(R)$ if and only if
		\[
		\sqrt{aR}=\sqrt{bR}.
		\]
		\item The map $\rho$ induces a homeomorphism
		\[
		KQ(GD(R))
		\cong
		\mathcal R_{\mathrm{pr}}(R)
		\]
		where $\mathcal R_{\mathrm{pr}}(R)$ carries the upper Alexandrov topology.
	\end{enumerate}
\end{theorem}

\begin{proof}
	$(1)$ If $[a]=[b]$, then $aR=bR$, and hence
	$\sqrt{aR}=\sqrt{bR}$. Thus $\rho$ is well defined, and surjectivity is
	immediate from the definition of $\mathcal R_{\mathrm{pr}}(R)$.
	
	By Lemma~\ref{lem:radical-divisibility},
	\[
	[b]\in B_a
	\Longleftrightarrow
	\sqrt{aR}\subseteq\sqrt{bR},
	\]
	which proves
	\[
	B_a=\rho^{-1}\bigl(\uparrow\sqrt{aR}\bigr).
	\]
	
	$(2)$ Since $B_a$ is the smallest open neighbourhood of $[a]$, two points are
	topologically indistinguishable if and only if their smallest open
	neighbourhoods coincide. Thus
	\[
	[a]\equiv[b]
	\Longleftrightarrow
	B_a=B_b.
	\]
	Using the preceding formula, or equivalently
	Lemma~\ref{lem:radical-divisibility}, this is equivalent to
	\[
	\sqrt{aR}=\sqrt{bR}.
	\]

	$(3)$ The fibres of $\rho$ are exactly the topological-indistinguishability
	classes. It follows that $\rho$ is constant on the equivalence classes of
	topological indistinguishability. Hence $\rho$ factors uniquely through
	the quotient map
	\[
	q:GD(R)\longrightarrow KQ(GD(R))
	\]
	to give a map
	\[
	\overline{\rho}:KQ(GD(R))
	\longrightarrow
	\mathcal R_{\mathrm{pr}}(R),
	\qquad
	\overline{\rho}([[a]])=\sqrt{aR}.
	\]
	By part~(2),
	\[
	[[a]]=[[b]]
	\quad\Longleftrightarrow\quad
	\sqrt{aR}=\sqrt{bR},
	\]
	so $\overline{\rho}$ is injective; its surjectivity follows from that of
	$\rho$. Thus $\overline{\rho}$ is a bijection.
	
	It remains to compare the two topologies. Since
	\[
	\rho=\overline{\rho}\circ q
	\]
	and
	\[
	B_a=\rho^{-1}\bigl(\uparrow\sqrt{aR}\bigr),
	\]
	we have
	\[
	B_a
	=
	q^{-1}\!\left(
	\overline{\rho}^{-1}
	\bigl(\uparrow\sqrt{aR}\bigr)
	\right).
	\]
	As $B_a$ is open in $GD(R)$, the definition of the quotient topology
	implies that
	\[
	\overline{\rho}^{-1}
	\bigl(\uparrow\sqrt{aR}\bigr)
	\]
	is open in $KQ(GD(R))$. Since the principal upper sets form a basis for
	the upper Alexandrov topology on $\mathcal R_{\mathrm{pr}}(R)$,
	$\overline{\rho}$ is continuous.
	
	Moreover, each $B_a$ is saturated with respect to topological
	indistinguishability. Indeed, if $[b]\in B_a$ and $[b]\equiv[c]$, then
	\[
	\sqrt{bR}=\sqrt{cR},
	\]
	and hence $[c]\in B_a$. Therefore,
	\[
	q^{-1}(q(B_a))=B_a,
	\]
	so $q(B_a)$ is open in $KQ(GD(R))$. Furthermore,
	\[
	\overline{\rho}(q(B_a))
	=
	\uparrow\sqrt{aR}.
	\]
	Thus $\overline{\rho}$ maps a basis of $KQ(GD(R))$ onto the basis of
	principal upper sets of $\mathcal R_{\mathrm{pr}}(R)$. Hence
	$\overline{\rho}$ is an open continuous bijection, and therefore a
	homeomorphism.
\end{proof}

\begin{corollary}\label{cor:closure}
	Let $R$ be a ring and $a\in R^\#$. Then
	\[
	\overline{\{[a]\}}
	=
	\{[b]\in EC(R^\#):
	\sqrt{bR}\subseteq\sqrt{aR}\}.
	\]
	Equivalently,
	\[
	\overline{\{[a]\}}
	=
	\{[b]\in EC(R^\#):
	a\mid b^n\text{ for some }n\geq1\}.
	\]
\end{corollary}

\begin{proof}
	In the upper Alexandrov topology on
	$\mathcal R_{\mathrm{pr}}(R)$, the closure of a point $P$ is
	\[
	\downarrow P=\{Q:Q\subseteq P\}.
	\]
	The result follows from Theorem~\ref{thm:KQ}. Equivalently, one may use
	Proposition~\ref{prop:basis}: $[b]\in\overline{\{[a]\}}$ if and only if
	the smallest open neighbourhood $B_b$ of $[b]$ contains $[a]$, which is
	equivalent to $a\mid b^n$ for some $n\geq1$.
\end{proof}

\section{Separation and comparison of divisor topologies}

The radical description obtained in the preceding section makes the
separation properties of $GD(R)$ particularly transparent. The possible
failure of the $T_0$-property is caused precisely by distinct principal
ideals having the same radical. We first record this observation and then
determine the rings for which this phenomenon does not occur.

We shall also use the following elementary fact concerning the ordinary
divisor topology. For integral domains, this was proved in
\cite[Proposition 2]{YigitKoc}; the same argument remains valid for
arbitrary commutative rings.

\begin{lemma}\label{lem:D-T0}
	Let $R$ be a ring. Then $D(R)$ is a $T_0$-space.
\end{lemma}

\begin{proof}
	Let $[a]\neq[b]$. Then $aR\neq bR$, so $a\nmid b$ or $b\nmid a$.
	If $a\nmid b$, then $[b]\in U_b$ while $[a]\notin U_b$.
	The other case is analogous. Hence $D(R)$ is $T_0$.
\end{proof}

\begin{proposition}\label{prop:T0}
	Let $R$ be a ring. The following statements are equivalent.
	\begin{enumerate}
		\item $GD(R)$ is a $T_0$-space.
		\item For all $a,b\in R^\#$,
		\[
		\sqrt{aR}=\sqrt{bR}
		\quad\Longrightarrow\quad
		aR=bR.
		\]
		\item The map
		\[
		\rho:EC(R^\#)\longrightarrow\mathcal R_{\mathrm{pr}}(R),
		\qquad
		\rho([a])=\sqrt{aR},
		\]
		is injective.
	\end{enumerate}
\end{proposition}

\begin{proof}
	By Theorem~\ref{thm:KQ}, two points $[a],[b]\in EC(R^\#)$ are
	topologically indistinguishable in $GD(R)$ if and only if
	\[
	\sqrt{aR}=\sqrt{bR}.
	\]
	Hence $GD(R)$ is $T_0$ if and only if equality of these radicals implies
	$[a]=[b]$, or equivalently $aR=bR$. This is precisely the injectivity of
	$\rho$.
\end{proof}

The preceding condition need not hold even for principal ideal domains.

\begin{example}\label{ex:Gaussian}
	Let $R=\mathbb Z[i]$. Since
	\[
	(1+i)\mid2
	\qquad\text{and}\qquad
	2\mid(1+i)^2,
	\]
	we have
	\[
	\sqrt{2R}=\sqrt{(1+i)R}.
	\]
	On the other hand,
	\[
	2R=(1+i)^2R\subsetneq(1+i)R,
	\]
	and hence
	\[
	[2]\neq[1+i].
	\]
	It follows from Proposition~\ref{prop:T0} that $GD(\mathbb Z[i])$ is not
	a $T_0$-space. In particular,
	\[
	GD(\mathbb Z[i])\neq D(\mathbb Z[i]).
	\]
\end{example}

We now give the ring-theoretic classification underlying
Proposition~\ref{prop:T0}. Recall that a commutative ring $R$ is
\emph{von Neumann regular} if for every $a\in R$ there exists $x\in R$
such that
\[
a=a^2x.
\]
Equivalently,
\[
aR=a^2R
\]
for every $a\in R$; see \cite{von,Brewer}.

\begin{theorem}\label{thm:T0-classification}
	Let $R$ be a ring. The following statements are equivalent.
	\begin{enumerate}
		\item $GD(R)$ is a $T_0$-space.
		
		\item For all $a,b\in R^\#$,
		\[
		\sqrt{aR}=\sqrt{bR}
		\quad\Longrightarrow\quad
		aR=bR.
		\]
		
		\item One of the following conditions holds:
		\begin{enumerate}
			\item $R$ is a von Neumann regular ring;
			\item $R$ is a local ring with nonzero maximal ideal
			\[
			\mathfrak m=\operatorname{Nil}(R)=aR
			\]
			for some $a\in R$, and $\mathfrak m^2=0$.
		\end{enumerate}
		
		\item $GD(R)=D(R)$.
	\end{enumerate}
\end{theorem}

\begin{proof}
	The equivalence $(1)\Leftrightarrow(2)$ follows from
	Proposition~\ref{prop:T0}.
	
	Assume (2). We consider two cases.
	
	Suppose first that $R$ is reduced. Let $a\in R^\#$. Since $a$ is
	nonnilpotent, $a^2\neq0$, and
	\[
	\sqrt{aR}=\sqrt{a^2R}.
	\]
	By (2),
	\[
	aR=a^2R.
	\]
	The same equality is trivial when $a=0$ or $a$ is a unit. Hence
	$aR=a^2R$ for every $a\in R$, and therefore $R$ is von Neumann
	regular.
	
	Now suppose that $R$ is not reduced. Choose
	\[
	0\neq a\in\operatorname{Nil}(R)
	\]
	such that $a^2=0$. For every
	$0\neq x\in\operatorname{Nil}(R)$,
	\[
	\sqrt{xR}
	=
	\operatorname{Nil}(R)
	=
	\sqrt{aR}.
	\]
	Thus (2) gives
	\[
	xR=aR.
	\]
	It follows that
	\[
	\operatorname{Nil}(R)=aR
	\qquad\text{and}\qquad
	\operatorname{Nil}(R)^2=0.
	\]
	
	We next show that $R$ has no nontrivial idempotents. Suppose that
	$e^2=e$ with $e\neq0,1$. Since
	\[
	a=ae+a(1-e),
	\]
	at least one of $ae$ and $a(1-e)$ is nonzero. They cannot both be
	nonzero, because then both are nonzero nilpotents and hence
	\[
	aeR=aR=a(1-e)R,
	\]
	which is impossible after multiplying by $e$. Replacing $e$ by $1-e$
	if necessary, we may assume that
	\[
	a=ae\neq 0
	\qquad\text{and}\qquad
	a(1-e)=0.
	\]
	Put
	\[
	x=a+(1-e).
	\]
	Then $x$ is a nonzero nonunit and
	\[
	x^2=1-e.
	\]
	Hence
	\[
	\sqrt{xR}=\sqrt{x^2R}=\sqrt{(1-e)R}.
	\]
	By (2),
	\[
	xR=(1-e)R.
	\]
	Multiplying by $e$ gives $a=0$, a contradiction. Thus $R$ is
	indecomposable.
	
	Let $b\in R\setminus\operatorname{Nil}(R)$. We claim that $b$ is a
	unit. Suppose otherwise. Then $b\in R^\#$ and $b^2\in R^\#$, and
	\[
	\sqrt{bR}=\sqrt{b^2R}.
	\]
	By (2),
	\[
	bR=b^2R.
	\]
	Thus $b=b^2c$ for some $c\in R$. Setting $e=bc$, we obtain
	\[
	e^2=e
	\qquad\text{and}\qquad
	bR=eR.
	\]
	Since $b$ is a nonzero nonunit, $e$ is a nontrivial idempotent,
	contrary to the preceding paragraph. Therefore every element outside
	$\operatorname{Nil}(R)$ is a unit. It follows that $R$ is local with
	unique maximal ideal
	\[
	\mathfrak m=\operatorname{Nil}(R)=aR,
	\qquad
	\mathfrak m^2=0.
	\]
	This proves $(2)\Rightarrow(3)$.
	
	Assume (3). If $R$ is von Neumann regular, then
	\[
	aR=a^nR
	\]
	for every $a\in R$ and every $n\geq1$. Hence, for $a\in R^\#$,
	\[
	[x]\in B_a
	\quad\Longleftrightarrow\quad
	x\mid a^n\text{ for some }n\geq1
	\quad\Longleftrightarrow\quad
	x\mid a,
	\]
	and therefore
	\[
	B_a=U_a.
	\]
	Thus $GD(R)=D(R)$.
	
	Suppose instead that $R$ is local with
	\[
	\mathfrak m=\operatorname{Nil}(R)=aR\neq0
	\qquad\text{and}\qquad
	\mathfrak m^2=0.
	\]
	Let $0\neq x\in R^\#$. Then $x\in\mathfrak m$, so $x=ra$ for some
	$r\in R$. If $r\in\mathfrak m$, then
	\[
	x=ra\in\mathfrak m^2=0,
	\]
	a contradiction. Hence $r$ is a unit, and therefore
	\[
	xR=aR.
	\]
	Thus
	\[
	EC(R^\#)=\{[a]\},
	\]
	and both $D(R)$ and $GD(R)$ are the topology on this singleton. Hence
	again $GD(R)=D(R)$. This proves $(3)\Rightarrow(4)$.
	
	Finally, if $GD(R)=D(R)$, then Lemma~\ref{lem:D-T0} implies that
	$GD(R)$ is $T_0$. Hence $(4)\Rightarrow(1)$.
\end{proof}

The nonreduced alternative in Theorem~\ref{thm:T0-classification} is
closely related to the class of UN-rings. Recall that a commutative ring
$R$ is a \emph{UN-ring} if every nonunit is a product of a unit and a
nilpotent element. Equivalently, $\operatorname{Nil}(R)$ is a maximal
ideal; see \cite{Calug,Mimouni}. Thus the nonreduced rings occurring in
Theorem~\ref{thm:T0-classification} form a particularly restricted class
of UN-rings.

\begin{example}\label{ex:square-zero-local}
	The nonreduced alternative of Theorem~\ref{thm:T0-classification}
	occurs naturally.
	\begin{enumerate}
		\item If $p$ is prime and $R=\mathbb Z_{p^2}$, then
		\[
		\operatorname{Nil}(R)=\overline pR,
		\qquad
		(\overline pR)^2=0.
		\]
		Hence
		\[
		EC(R^\#)=\{[\overline p]\}
		\]
		and $GD(R)=D(R)$.
		
		\item Let $k$ be a field and
		\[
		R=k[[X]]/(X^2).
		\]
		If $x=X+(X^2)$, then
		\[
		\operatorname{Nil}(R)=xR,
		\qquad
		x^2=0.
		\]
		Again $EC(R^\#)=\{[x]\}$ and $GD(R)=D(R)$.
	\end{enumerate}
\end{example}

The classification also has the following ideal-theoretic consequence.

\begin{corollary}\label{cor:radical-ideals}
	Suppose that $GD(R)$ is a $T_0$-space. Then every nonzero ideal of $R$
	is radical. Consequently, for every nonzero proper ideal $I$ of $R$,
	the quotient $R/I$ is von Neumann regular.
\end{corollary}

\begin{proof}
	By Theorem~\ref{thm:T0-classification}, either $R$ is von Neumann
	regular or $R$ is local with
	\[
	\mathfrak m=\operatorname{Nil}(R)=aR\neq0,
	\qquad
	\mathfrak m^2=0.
	\]
	In the first case every ideal of $R$ is radical.
	
	In the second case, let $0\neq I\subsetneq R$. Choose $0\neq x\in I$.
	Then $x\in\mathfrak m$, so $x=ra$ for some $r\in R$. As above,
	$x\neq0$ implies that $r$ is a unit. Hence
	\[
	xR=aR=\mathfrak m,
	\]
	and therefore $I=\mathfrak m$. Thus the only nonzero proper ideal is
	$\mathfrak m$, which is radical.
	
	Now let $I$ be a nonzero proper ideal of $R$. Every ideal of $R/I$
	has the form $J/I$ for some ideal $J$ containing $I$. Since $J$ is
	nonzero, it is radical. Hence every ideal of $R/I$ is radical, and
	therefore $R/I$ is von Neumann regular.
\end{proof}

We next turn to stronger separation axioms. Since $GD(R)$ is Alexandrov,
the $T_1$ condition forces discreteness.

\begin{theorem}\label{thm:discrete}
	Let $R$ be a ring. The following statements are equivalent.
	\begin{enumerate}
		\item $GD(R)$ is a $T_1$-space.
		\item $GD(R)$ is Hausdorff.
		\item $GD(R)$ is discrete.
		\item One of the following conditions holds:
		\begin{enumerate}
			\item $R$ is a field;
			\item $R=K\times F$ for some fields $K$ and $F$;
			\item $R$ is a local ring with nonzero maximal ideal
			\[
			\mathfrak m=\operatorname{Nil}(R)=aR
			\]
			and $\mathfrak m^2=0$.
		\end{enumerate}
	\end{enumerate}
\end{theorem}

\begin{proof}
	Every discrete space is Hausdorff and every Hausdorff space is $T_1$.
	Conversely, every Alexandrov $T_1$-space is discrete. Since $GD(R)$ is
	Alexandrov by Proposition~\ref{prop:basis}, conditions $(1)$--$(3)$ are
	equivalent.
	
	Assume that $GD(R)$ is discrete. Then it is $T_0$, so
	Theorem~\ref{thm:T0-classification} applies. If $R$ is in the
	nonreduced local case of that theorem, condition $(4)(c)$ holds.
	
	Suppose that $R$ is von Neumann regular. If $R$ is indecomposable, then
	a commutative von Neumann regular ring with no nontrivial idempotents is a
	field, giving $(4)(a)$.
	
	Assume that $R$ is decomposable. Then
	\[
	R=R_1\times R_2
	\]
	for some nonzero rings $R_1$ and $R_2$. We show that $R_1$ is a field.
	Let $0\neq x\in R_1$. The element $(x,0)$ is a nonzero nonunit of $R$.
	Since
	\[
	(x,0)=(x,1)(1,0),
	\]
	we have
	\[
	[(x,1)]\in U_{(x,0)}\subseteq B_{(x,0)}
	\]
	unless $(x,1)$ is a unit. Since $GD(R)$ is discrete,
	\[
	B_{(x,0)}=\{[(x,0)]\}.
	\]
	But
	\[
	R(x,1)\neq R(x,0),
	\]
	so $(x,1)$ must be a unit. Hence $x$ is a unit of $R_1$. Thus $R_1$
	is a field. Similarly, $R_2$ is a field, proving $(4)(b)$.
	
	Conversely, if $R$ is a field, then $EC(R^\#)=\emptyset$, so $GD(R)$
	is discrete. If $R=K\times F$ for fields $K$ and $F$, then
	\[
	EC(R^\#)
	=
	\{[(1,0)],[(0,1)]\},
	\]
	and
	\[
	B_{(1,0)}=\{[(1,0)]\},
	\qquad
	B_{(0,1)}=\{[(0,1)]\}.
	\]
	Hence $GD(R)$ is discrete.
	
	Finally, in the square-zero local case,
	\[
	EC(R^\#)=\{[a]\}
	\]
	by the proof of Theorem~\ref{thm:T0-classification}, so $GD(R)$ is
	again discrete.
\end{proof}

At the opposite extreme, $GD(R)$ may contain no nontrivial proper open
sets. The following result characterizes this situation.

\begin{theorem}\label{thm:indiscrete}
	Let $R$ be a ring. Then $GD(R)$ is indiscrete if and only if one of the
	following conditions holds.
	\begin{enumerate}
		\item $R$ is a UN-ring which is not an integral domain;
		\item $R$ is an integral domain with at most one nonzero prime ideal.
	\end{enumerate}
	In the second case, if $R$ is not a field, then $R$ is a local domain
	with a unique nonzero prime ideal.
\end{theorem}

\begin{proof}
	Assume first that $GD(R)$ is indiscrete.
	
	Suppose that $R$ is not an integral domain. We claim that every
	nonzero nonunit of $R$ is nilpotent.
	
	If $R$ is not reduced, choose
	\[
	0\neq x\in\operatorname{Nil}(R).
	\]
	For any $a\in R^\#$, indiscreteness gives
	\[
	B_a=EC(R^\#),
	\]
	and hence $[x]\in B_a$. By Lemma~\ref{lem:radical-divisibility},
	\[
	\sqrt{aR}\subseteq\sqrt{xR}
	=
	\operatorname{Nil}(R).
	\]
	Thus $a\in\operatorname{Nil}(R)$.
	
	It remains to rule out the case in which $R$ is reduced but not a
	domain. Let $P$ be a minimal prime ideal of $R$. Since $R$ is reduced
	and not a domain, every minimal prime is nonzero. Choose
	$0\neq x\in P$. For any $a\in R^\#$, again $[x]\in B_a$, and therefore
	\[
	\sqrt{aR}\subseteq\sqrt{xR}\subseteq P.
	\]
	Hence $a\in P$. Repeating this for every minimal prime gives
	\[
	a\in\bigcap_{P\in\operatorname{Min}(R)}P
	=
	\operatorname{Nil}(R)=0,
	\]
	a contradiction. Thus the reduced non-domain case cannot occur.
	
	Therefore, when $R$ is not a domain, every nonzero nonunit is nilpotent.
	Equivalently, $\operatorname{Nil}(R)$ is the unique maximal ideal, and
	hence $R$ is a UN-ring.
	
	Now suppose that $R$ is an integral domain. If $R$ is a field, there is
	nothing to prove. Assume that $R$ is not a field. If $P$ and $Q$ are
	distinct nonzero prime ideals, then, after interchanging them if
	necessary, choose
	\[
	a\in P\setminus Q
	\qquad\text{and}\qquad
	0\neq b\in Q.
	\]
	Indiscreteness gives
	\[
	B_a=B_b=EC(R^\#),
	\]
	so by Lemma~\ref{lem:radical-divisibility},
	\[
	\sqrt{aR}=\sqrt{bR}.
	\]
	But
	\[
	\sqrt{bR}\subseteq Q,
	\]
	and hence $a\in Q$, a contradiction. Thus $R$ has a unique nonzero
	prime ideal.
	
	Conversely, suppose that $R$ is a UN-ring which is not a domain. Then
	every nonzero nonunit is nilpotent. By
	Proposition~\ref{prop:D-GD},
	\[
	B_a=EC(R^\#)
	\]
	for every $a\in R^\#$, and hence $GD(R)$ is indiscrete.
	
	Finally, let $R$ be an integral domain with at most one nonzero prime
	ideal. If $R$ is a field, then $EC(R^\#)=\emptyset$, and the conclusion
	is immediate. Otherwise let $P$ be the unique nonzero prime ideal. For
	every $a\in R^\#$,
	\[
	\sqrt{aR}=P.
	\]
	Hence all basic open sets $B_a$ coincide. Since they cover
	$EC(R^\#)$, each of them is equal to the whole space. Therefore
	$GD(R)$ is indiscrete.
\end{proof}

\section{Isolated points and arithmetic applications}

We now study the isolated points of the generalized divisor topology.
We use the terminology from the factorization theory of rings with zero
divisors. Following \cite{AndersonValdesLeon,AndersonChunII}, a nonunit
$a\in R$ is called \emph{$m$-irreducible} if $aR$ is maximal in the set
of proper principal ideals of $R$. Thus, for $a\in R^\#$,
$a$ is $m$-irreducible if and only if
\[
aR\subseteq bR\subsetneq R
\quad\Longrightarrow\quad
aR=bR.
\]
For integral domains, $m$-irreducibility agrees with the usual notion of
irreducibility.

Recall also that an element $a\in R$ is called a
\emph{von Neumann regular element} if
\[
a=a^2x
\]
for some $x\in R$, equivalently if
\[
aR=a^2R.
\]

Since $B_a$ is the smallest open neighbourhood of $[a]$, the point
$[a]$ is isolated precisely when
\[
B_a=\{[a]\}.
\]
The following result gives an algebraic characterization of this
condition.

\begin{proposition}\label{prop:isolated}
	Let $R$ be a ring and let $[a]\in EC(R^\#)$. Then $[a]$ is an isolated
	point of $GD(R)$ if and only if one of the following conditions holds.
	\begin{enumerate}
		\item $R$ is a local ring with maximal ideal
		\[
		\mathfrak m=\operatorname{Nil}(R)=aR
		\]
		such that $a^2=0$. In this case,
		\[
		EC(R^\#)=\{[a]\}.
		\]
		
		\item $a$ is an $m$-irreducible von Neumann regular element.
	\end{enumerate}
\end{proposition}

\begin{proof}
	Suppose first that $[a]$ is isolated. Since $B_a$ is the smallest open
	neighbourhood of $[a]$, we have
	\[
	B_a=\{[a]\}.
	\]
	
	Assume first that $a$ is nilpotent. By
	Proposition~\ref{prop:D-GD},
	\[
	B_a=EC(R^\#),
	\]
	and hence
	\[
	EC(R^\#)=\{[a]\}.
	\]
	We claim that $a^2=0$. Indeed, if $a^2\neq0$, then
	$a^2\in R^\#$ and
	\[
	[a^2]\in B_a=\{[a]\}.
	\]
	Thus
	\[
	aR=a^2R.
	\]
	It follows successively that
	\[
	aR=a^nR
	\]
	for every $n\geq1$. Since $a$ is nilpotent, this eventually gives
	$aR=0$, contradicting $a\neq0$. Hence $a^2=0$.
	
	Now let $0\neq x\in R^\#$. Since
	$EC(R^\#)=\{[a]\}$, we have
	\[
	xR=aR.
	\]
	Thus every nonzero nonunit belongs to $aR$, and every element of $aR$
	is nilpotent since $a^2=0$. It follows that
	\[
	\operatorname{Nil}(R)=aR,
	\]
	and every element outside $aR$ is a unit. Hence $R$ is local with
	unique maximal ideal
	\[
	\mathfrak m=\operatorname{Nil}(R)=aR.
	\]
	This proves~(1).
	
	Now suppose that $a$ is not nilpotent. Then $a^2\in R^\#$ and
	\[
	[a^2]\in B_a=\{[a]\}.
	\]
	Hence
	\[
	aR=a^2R,
	\]
	so $a$ is a von Neumann regular element.
	
	Moreover,
	\[
	U_a\subseteq B_a=\{[a]\},
	\]
	and therefore
	\[
	U_a=\{[a]\}.
	\]
	Let
	\[
	aR\subseteq bR\subsetneq R.
	\]
	Since $aR\neq0$, we have $b\neq0$, and since $bR\neq R$, the element
	$b$ is a nonunit. Thus $b\in R^\#$ and $b\mid a$, so
	\[
	[b]\in U_a=\{[a]\}.
	\]
	Consequently,
	\[
	aR=bR.
	\]
	Hence $a$ is $m$-irreducible, proving~(2).
	
	Conversely, suppose that condition~(1) holds. Let
	$0\neq x\in R^\#$. Since $x\in\mathfrak m=aR$, write
	\[
	x=ra
	\]
	for some $r\in R$. If $r\in\mathfrak m$, then
	\[
	x=ra\in\mathfrak m^2=0,
	\]
	a contradiction. Hence $r$ is a unit, and therefore
	\[
	xR=aR.
	\]
	Thus
	\[
	EC(R^\#)=\{[a]\},
	\]
	and $[a]$ is isolated.
	
	Finally, suppose that $a$ is $m$-irreducible and von Neumann regular,
	and let $[b]\in B_a$. Then
	\[
	b\mid a^n
	\]
	for some $n\geq1$. Since
	\[
	aR=a^2R,
	\]
	we have
	\[
	aR=a^nR
	\]
	for every $n\geq1$. Hence
	\[
	aR=a^nR\subseteq bR.
	\]
	Since $b\in R^\#$, the ideal $bR$ is proper. By the
	$m$-irreducibility of $a$,
	\[
	aR=bR.
	\]
	Thus $[b]=[a]$, and hence
	\[
	B_a=\{[a]\}.
	\]
	Therefore $[a]$ is isolated.
\end{proof}

The preceding characterization shows a sharp difference between the
ordinary and generalized divisor topologies over integral domains.

\begin{corollary}\label{cor:no-isolated-domain}
	Let $R$ be an integral domain. Then $GD(R)$ has no isolated points.
\end{corollary}

\begin{proof}
	Suppose that $[a]$ is isolated. Since a domain has no nonzero nilpotent
	elements, Proposition~\ref{prop:isolated} implies that $a$ is a
	von Neumann regular element. Hence
	\[
	a=a^2x
	\]
	for some $x\in R$. Since $a\neq0$, cancellation gives
	\[
	1=ax,
	\]
	so $a$ is a unit, a contradiction.
\end{proof}

\begin{remark}
	For an integral domain, the behavior of the ordinary divisor topology is
	different. By \cite[Proposition 3]{YigitKoc}, the isolated points of
	$D(R)$ are precisely the classes of irreducible elements. Since
	$m$-irreducibility and ordinary irreducibility coincide in an integral
	domain, Proposition~\ref{prop:isolated} shows that the additional
	von Neumann regularity condition is precisely what eliminates isolated
	points from $GD(R)$.
\end{remark}

We next specialize the preceding characterization to principal ideal
rings.

\begin{lemma}\label{lem:m-irred-pir}
	Let $R$ be a principal ideal ring and let $a\in R^\#$. Then $a$ is
	$m$-irreducible if and only if $aR$ is a nonzero maximal ideal of $R$.
\end{lemma}

\begin{proof}
	Since every ideal of $R$ is principal, the ideal $aR$ is maximal among
	the proper principal ideals if and only if it is maximal among all proper
	ideals. Since $a\neq0$, this is equivalent to $aR$ being a nonzero
	maximal ideal.
\end{proof}

For a topological space $X$, let $\operatorname{Iso}(X)$ denote the set
of its isolated points.

\begin{theorem}\label{thm:Zn-isolated}
	Let
	\[
	n=p_1^{m_1}\cdots p_k^{m_k},
	\]
	where $p_1,\ldots,p_k$ are distinct prime numbers and $m_i\geq1$.
	Then the following statements hold.
	\begin{enumerate}
		\item If $k=1$, then
		\[
		\operatorname{Iso}(GD(\mathbb Z_n))
		=
		\begin{cases}
			\{[\overline{p_1}]\}, & m_1=2,\\[2mm]
			\emptyset, & m_1\neq2.
		\end{cases}
		\]
		
		\item If $k\geq2$, then
		\[
		\operatorname{Iso}(GD(\mathbb Z_n))
		=
		\{[\overline{p_i}]:m_i=1\}.
		\]
		
		\item For the ordinary divisor topology,
		\[
		\operatorname{Iso}(D(\mathbb Z_n))
		=
		\begin{cases}
			\emptyset, & n \text{ is prime},\\[2mm]
			\{[\overline{p_i}]:1\leq i\leq k\}, & n \text{ is composite}.
		\end{cases}
		\]
	\end{enumerate}
\end{theorem}

\begin{proof}
	Suppose first that $k=1$, so that
	\[
	n=p_1^{m_1}.
	\]
	If $m_1=1$, then $\mathbb Z_n$ is a field and hence
	\[
	EC(\mathbb Z_n^\#)=\emptyset.
	\]
	Assume that $m_1\geq2$. Then $\mathbb Z_n$ is local with
	\[
	\operatorname{Nil}(\mathbb Z_n)
	=
	\overline{p_1}\mathbb Z_n.
	\]
	By Proposition~\ref{prop:isolated}, the nilpotent alternative produces
	an isolated point precisely when
	\[
	\overline{p_1}^{\,2}=0,
	\]
	which is equivalent to $m_1=2$.
	
	If $m_1>2$, the only possible $m$-irreducible class is
	$[\overline{p_1}]$, since
	$\overline{p_1}\mathbb Z_n$ is the unique maximal ideal. However,
	\[
	\overline{p_1}\mathbb Z_n
	\neq
	\overline{p_1}^{\,2}\mathbb Z_n,
	\]
	so $\overline{p_1}$ is not von Neumann regular. Thus no isolated point
	occurs. This proves~(1).
	
	Now assume that $k\geq2$. Since $\mathbb Z_n$ is a principal ideal
	ring, Lemma~\ref{lem:m-irred-pir} shows that its $m$-irreducible
	classes are precisely
	\[
	[\overline{p_1}],\ldots,[\overline{p_k}].
	\]
	For each $i$,
	\[
	\overline{p_i}\mathbb Z_n
	=
	\overline{p_i}^{\,2}\mathbb Z_n
	\]
	if and only if
	\[
	\gcd(p_i,n)=\gcd(p_i^2,n).
	\]
	Since
	\[
	\gcd(p_i,n)=p_i,
	\]
	this equality holds exactly when $m_i=1$. By
	Proposition~\ref{prop:isolated},
	\[
	\operatorname{Iso}(GD(\mathbb Z_n))
	=
	\{[\overline{p_i}]:m_i=1\}.
	\]
	This proves~(2).
	
	Finally, a point $[\overline a]$ is isolated in $D(\mathbb Z_n)$ if
	and only if
	\[
	U_{\overline a}=\{[\overline a]\}.
	\]
	This is equivalent to $\overline a$ being $m$-irreducible. By
	Lemma~\ref{lem:m-irred-pir}, these are precisely the generators of the
	nonzero maximal ideals of $\mathbb Z_n$, namely
	\[
	\overline{p_i}\mathbb Z_n,
	\qquad
	1\leq i\leq k.
	\]
	Hence
	\[
	\operatorname{Iso}(D(\mathbb Z_n))
	=
	\{[\overline{p_i}]:1\leq i\leq k\}.
	\]
\end{proof}

\begin{example}\label{ex:Z60}
	Let
	\[
	n=60=2^2\cdot3\cdot5.
	\]
	Then Theorem~\ref{thm:Zn-isolated} gives
	\[
	\operatorname{Iso}(GD(\mathbb Z_{60}))
	=
	\{[\overline3],[\overline5]\},
	\]
	whereas
	\[
	\operatorname{Iso}(D(\mathbb Z_{60}))
	=
	\{[\overline2],[\overline3],[\overline5]\}.
	\]
	Thus $[\overline2]$ is isolated in $D(\mathbb Z_{60})$ but not in
	$GD(\mathbb Z_{60})$.
\end{example}

\section{Order and finiteness properties}

The order-theoretic description of $GD(R)$ developed in Section~1
allows several global topological properties to be expressed in
algebraic terms. We first consider nestedness and then turn to
compactness, the Lindel\"of property, and Noetherianity.

\subsection{Nestedness}

Recall that a topological space is called \emph{nested} if its open
sets are linearly ordered by inclusion. For a space with a basis, it is
enough to require the basis elements to be linearly ordered; see
\cite[Lemma 1]{YigitKoc}.

Following Fuchs \cite{FuchsQuasiPrimary}, a proper ideal $I$ of a ring
$R$ is called \emph{quasi-primary} if $\sqrt{I}$ is a prime ideal.

\begin{theorem}\label{thm:nested}
	Let $R$ be a ring. The following statements are equivalent.
	\begin{enumerate}
		\item $GD(R)$ is a nested space.
		
		\item The poset
		\[
		\mathcal R_{\mathrm{pr}}(R)
		=
		\{\sqrt{aR}:a\in R^\#\}
		\]
		is linearly ordered by inclusion.
		
		\item The radical ideals of $R$ are linearly ordered by inclusion.
		
		\item The prime ideals of $R$ are linearly ordered by inclusion.
		
		\item Every proper ideal of $R$ is quasi-primary.
		
		\item The Zariski topology on $\operatorname{Spec}(R)$ is nested.
	\end{enumerate}
\end{theorem}

\begin{proof}
	By Lemma~\ref{lem:radical-divisibility},
	\[
	B_a\subseteq B_b
	\quad\Longleftrightarrow\quad
	\sqrt{bR}\subseteq\sqrt{aR}.
	\]
	Hence the basic open sets of $GD(R)$ are linearly ordered by inclusion
	if and only if $\mathcal R_{\mathrm{pr}}(R)$ is linearly ordered by
	inclusion. By the basis criterion for nested spaces,
	$(1)\Leftrightarrow(2)$.
	
	Assume (2), and let $I$ and $J$ be two radical ideals of $R$.
	Suppose that they are not comparable. Choose
	\[
	x\in I\setminus J
	\qquad\text{and}\qquad
	y\in J\setminus I.
	\]
	Then $x,y\in R^\#$. By (2), either
	\[
	\sqrt{xR}\subseteq\sqrt{yR}
	\qquad\text{or}\qquad
	\sqrt{yR}\subseteq\sqrt{xR}.
	\]
	In the first case,
	\[
	x\in\sqrt{xR}\subseteq\sqrt{yR}\subseteq J,
	\]
	a contradiction. The second case similarly gives $y\in I$.
	Thus the radical ideals of $R$ are linearly ordered, proving
	$(2)\Rightarrow(3)$.
	
	The implication $(3)\Rightarrow(4)$ is immediate, since every prime
	ideal is radical.
	
	Assume (4), and let $I$ be a proper ideal of $R$. Since
	\[
	\sqrt{I}
	=
	\bigcap_{P\in V(I)}P
	\]
	and the prime ideals of $R$ are linearly ordered, $\sqrt{I}$ is prime.
	Indeed, suppose that
	\[
	ab\in\sqrt{I}
	\qquad\text{and}\qquad
	a,b\notin\sqrt{I}.
	\]
	Then there exist $P,Q\in V(I)$ such that
	\[
	a\notin P
	\qquad\text{and}\qquad
	b\notin Q.
	\]
	If, say, $P\subseteq Q$, then $b\notin P$, and hence
	$ab\notin P$, a contradiction. The case $Q\subseteq P$ is analogous.
	Therefore $\sqrt{I}$ is prime, so $I$ is quasi-primary. This proves
	$(4)\Rightarrow(5)$.
	
	Assume (5), and let $P,Q\in\operatorname{Spec}(R)$. The ideal
	$P\cap Q$ is proper and hence quasi-primary. Since $P\cap Q$ is
	radical,
	\[
	\sqrt{P\cap Q}=P\cap Q
	\]
	is prime. If neither $P\subseteq Q$ nor $Q\subseteq P$, choose
	\[
	x\in P\setminus Q
	\qquad\text{and}\qquad
	y\in Q\setminus P.
	\]
	Then
	\[
	xy\in P\cap Q,
	\]
	while neither $x$ nor $y$ belongs to $P\cap Q$, contradicting the
	primality of $P\cap Q$. Hence $P$ and $Q$ are comparable. Thus
	$(5)\Rightarrow(4)$.
	
	Assume (4). Since the prime ideals of $R$ are linearly ordered, any
	two Zariski closed subsets of $\operatorname{Spec}(R)$ are comparable.
	Indeed, each closed set $V(I)$ is an upper set in
	$\operatorname{Spec}(R)$ with respect to inclusion, and upper sets of
	a linearly ordered set are linearly ordered by inclusion. Hence the
	Zariski topology on $\operatorname{Spec}(R)$ is nested. This proves
	$(4)\Rightarrow(6)$.
	
	Finally, assume (6). For $a,b\in R^\#$, the closed sets
	$V(aR)$ and $V(bR)$ are comparable. Since
	\[
	V(aR)\subseteq V(bR)
	\quad\Longleftrightarrow\quad
	\sqrt{bR}\subseteq\sqrt{aR},
	\]
	the radicals $\sqrt{aR}$ and $\sqrt{bR}$ are comparable. Hence
	$(6)\Rightarrow(2)$.
\end{proof}

\begin{remark}
	For integral domains, the ordinary divisor topology behaves differently.
	By \cite[Theorem 2]{YigitKoc}, $D(R)$ is nested if and only if $R$ is
	a valuation domain. Theorem~\ref{thm:nested} shows that nestedness of
	$GD(R)$ is instead governed by the ordering of the prime spectrum.
\end{remark}

\subsection{Compactness, Lindel\"ofness, and Noetherianity}

We begin with an order-theoretic description of the covering properties
of $GD(R)$. 
Recall that a subset $C$ of a partially ordered set $P$ is called
\emph{coinitial} if for every $x\in P$ there exists $c\in C$ such that
$c\leq x$; see \cite[p.~46]{Roman}.

\begin{theorem}\label{thm:covering-coinitial}
	Let $R$ be a ring.
	\begin{enumerate}
		\item $GD(R)$ is compact if and only if
		$\mathcal R_{\mathrm{pr}}(R)$ has a finite coinitial subset.
		
		\item $GD(R)$ is Lindel\"of if and only if
		$\mathcal R_{\mathrm{pr}}(R)$ has a countable coinitial subset.
	\end{enumerate}
\end{theorem}

\begin{proof}
	Let $\{a_i:i\in\Lambda\}\subseteq R^\#$. Then
	\[
	EC(R^\#)=\bigcup_{i\in\Lambda}B_{a_i}
	\]
	if and only if for every $x\in R^\#$ there exists $i\in\Lambda$ such
	that
	\[
	[x]\in B_{a_i}.
	\]
	By Lemma~\ref{lem:radical-divisibility}, this is equivalent to
	\[
	\sqrt{a_iR}\subseteq\sqrt{xR}
	\]
	for some $i\in\Lambda$. Thus the family
	$\{B_{a_i}:i\in\Lambda\}$ covers $GD(R)$ if and only if
	\[
	\{\sqrt{a_iR}:i\in\Lambda\}
	\]
	is coinitial in $\mathcal R_{\mathrm{pr}}(R)$.
	
	Suppose first that $GD(R)$ is compact. The basic open cover
	\[
	EC(R^\#)=\bigcup_{a\in R^\#}B_a
	\]
	has a finite subcover
	\[
	EC(R^\#)
	=
	B_{a_1}\cup\cdots\cup B_{a_n}.
	\]
	Hence
	\[
	\{\sqrt{a_1R},\ldots,\sqrt{a_nR}\}
	\]
	is a finite coinitial subset of $\mathcal R_{\mathrm{pr}}(R)$.
	
Conversely, suppose that
\[
\{\sqrt{a_1R},\ldots,\sqrt{a_nR}\}
\]
is coinitial, and let $\mathcal U$ be an open cover of $GD(R)$.
For each $i$, choose $U_i\in\mathcal U$ such that $[a_i]\in U_i$.
Since $B_{a_i}$ is the smallest open neighbourhood of $[a_i]$,
\[
B_{a_i}\subseteq U_i.
\]
Moreover, coinitiality implies that for every $[x]\in EC(R^\#)$
there exists $i$ such that
\[
\sqrt{a_iR}\subseteq\sqrt{xR},
\]
and hence $[x]\in B_{a_i}$ by
Lemma~\ref{lem:radical-divisibility}. Therefore
\[
EC(R^\#)=\bigcup_{i=1}^n B_{a_i}
\subseteq
\bigcup_{i=1}^n U_i.
\]
Thus $\{U_1,\ldots,U_n\}$ forms a finite subcover. This proves~(1).
	
	The proof of~(2) is identical, with ``finite'' replaced by
	``countable''.
\end{proof}

For integral domains, finite coinitiality admits a classical
ring-theoretic interpretation. $G$-domains arise naturally in the
theory of Hilbert rings and Hilbert's Nullstellensatz, going back to
Goldman's work \cite{Goldman}; see also Kaplansky's treatment
\cite{Kaplansky}.

Let $R$ be a nonfield integral domain with quotient field $K$. Recall
that $R$ is called a \emph{$G$-domain} if $K$ is finitely generated as
an $R$-algebra. Equivalently,
\[
\bigcap_{\substack{P\in\operatorname{Spec}(R)\\P\neq(0)}}P
\neq(0).
\]
The next theorem identifies compactness of $GD(R)$ with this classical
ring-theoretic condition.

\begin{theorem}\label{thm:compact-G-domain}
	Let $R$ be a nonfield integral domain with quotient field $K$. Then the
	following statements are equivalent.
	\begin{enumerate}
		\item $GD(R)$ is compact.
		
		\item There exists $a\in R^\#$ such that
		\[
		B_a=EC(R^\#).
		\]
		
		\item
		\[
		\bigcap_{\substack{P\in\operatorname{Spec}(R)\\P\neq(0)}}P
		\neq(0).
		\]
		
		\item $K$ is finitely generated as an $R$-algebra; equivalently,
		$R$ is a $G$-domain.
	\end{enumerate}
\end{theorem}

\begin{proof}
	$(1)\Rightarrow(2)$
	By Theorem~\ref{thm:covering-coinitial}, there exist
	$a_1,\ldots,a_n\in R^\#$ such that
	\[
	\{\sqrt{a_1R},\ldots,\sqrt{a_nR}\}
	\]
	is coinitial in $\mathcal R_{\mathrm{pr}}(R)$. Put
	\[
	a=a_1\cdots a_n.
	\]
	Since $R$ is a domain, $a\in R^\#$. Moreover,
	\[
	\sqrt{aR}
	=
	\bigcap_{i=1}^n\sqrt{a_iR}.
	\]
	Let $x\in R^\#$. By coinitiality, there exists $i$ such that
	\[
	\sqrt{a_iR}\subseteq\sqrt{xR}.
	\]
	Hence
	\[
	\sqrt{aR}\subseteq\sqrt{xR},
	\]
	and Lemma~\ref{lem:radical-divisibility} gives
	\[
	[x]\in B_a.
	\]
	Therefore
	\[
	B_a=EC(R^\#).
	\]
	
	$(2)\Rightarrow(1)$ is clear.
	
	$(2)\Rightarrow(3)$
	Let $P$ be a nonzero prime ideal of $R$, and choose
	$0\neq x\in P$. Then $x\in R^\#$ and, by (2),
	\[
	[x]\in B_a.
	\]
	Thus
	\[
	\sqrt{aR}\subseteq\sqrt{xR}\subseteq P.
	\]
	In particular, $a\in P$. Since $P$ was arbitrary,
	\[
	0\neq a\in
	\bigcap_{\substack{P\in\operatorname{Spec}(R)\\P\neq(0)}}P.
	\]
	
	$(3)\Rightarrow(2)$
	Choose
	\[
	0\neq a\in
	\bigcap_{\substack{P\in\operatorname{Spec}(R)\\P\neq(0)}}P.
	\]
	Since $R$ is not a field, $a$ is a nonunit, and hence $a\in R^\#$.
	Let $x\in R^\#$. Every prime ideal containing $xR$ is nonzero, and
	therefore contains $a$. Hence
	\[
	a\in\sqrt{xR},
	\]
	so
	\[
	\sqrt{aR}\subseteq\sqrt{xR}.
	\]
	By Lemma~\ref{lem:radical-divisibility},
	\[
	[x]\in B_a.
	\]
	Thus
	\[
	B_a=EC(R^\#).
	\]
	
	Finally, $(3)\Leftrightarrow(4)$ is the classical characterization of
	$G$-domains; see \cite{Kaplansky}.
\end{proof}

\begin{remark}
	Theorem~\ref{thm:compact-G-domain} identifies a topological property of
	$GD(R)$ with a classical finiteness condition in commutative algebra:
	\[
	GD(R)\text{ is compact}
	\quad\Longleftrightarrow\quad
	\operatorname{Frac}(R)\text{ is finitely generated as an }R\text{-algebra}.
	\]
\end{remark}

We next consider the Lindel\"of property. For UFDs, the countable
coinitial condition has a simple factorization-theoretic interpretation.

\begin{corollary}\label{cor:UFD-Lindelof}
	Let $R$ be a UFD. Then $GD(R)$ is Lindel\"of if and only if $R$ has
	at most countably many nonassociate prime elements.
\end{corollary}

\begin{proof}
	Suppose first that $GD(R)$ is Lindel\"of. By
	Theorem~\ref{thm:covering-coinitial}, there exist
	$a_1,a_2,\ldots\in R^\#$ such that
	\[
	\{\sqrt{a_iR}:i\geq1\}
	\]
	is coinitial in $\mathcal R_{\mathrm{pr}}(R)$.
	
	Let $p$ be a prime element of $R$. By coinitiality, there exists $i$
	such that
	\[
	\sqrt{a_iR}\subseteq pR.
	\]
	In particular,
	\[
	a_i\in pR,
	\]
	so $p\mid a_i$. Since each $a_i$ has only finitely many nonassociate
	prime divisors, the set of nonassociate prime elements of $R$ is a
	countable union of finite sets. Hence it is at most countable.
	
	Conversely, suppose that $R$ has at most countably many nonassociate
	prime elements. Every nonzero nonunit of $R$ has a factorization
	involving finitely many of these prime elements. It follows that
	$EC(R^\#)$ is at most countable. Hence every open cover of $GD(R)$
	has a countable subcover, and $GD(R)$ is Lindel\"of.
\end{proof}

\begin{remark}\label{rem:Lindelof-cardinality}
	The Lindel\"of property does not impose a corresponding cardinality
	restriction on $EC(R^\#)$ in arbitrary rings. Let $k$ be an
	uncountable field and put
	\[
	R=k[x,y]/(x,y)^2.
	\]
	Then $R$ is local with maximal ideal
	\[
	\mathfrak m=(\overline{x},\overline{y}),
	\]
	and
	\[
	\mathfrak m^2=0.
	\]
	Hence every nonzero nonunit of $R$ is nilpotent, so
	\[
	B_a=EC(R^\#)
	\]
	for every $a\in R^\#$. Thus $GD(R)$ is indiscrete, and therefore
	compact and Lindel\"of.
	
	Moreover, $\mathfrak m$ is a two-dimensional vector space over $k$.
	For every $0\neq a\in\mathfrak m$, we have
	\[
	aR=ka,
	\]
	since $\mathfrak m^2=0$. Consequently,
	\[
	[a]=[b]
	\quad\Longleftrightarrow\quad
	ka=kb
	\]
	for $0\neq a,b\in\mathfrak m$. Hence the elements of $EC(R^\#)$ are
	precisely the one-dimensional $k$-subspaces of $\mathfrak m$, and
	therefore
	\[
	EC(R^\#)\cong\mathbb P^1(k).
	\]
	Since $k$ is uncountable, $EC(R^\#)$ is uncountable.
\end{remark}

\begin{example}\label{ex:Lindelof-not-compact}
	Let
	\[
	R=\mathbb Q[x_1,x_2,\ldots].
	\]
	Then $R$ is a UFD with at most countably many nonassociate prime
	elements. Hence $GD(R)$ is Lindel\"of by
	Corollary~\ref{cor:UFD-Lindelof}.
	
	On the other hand,
	\[
	\bigcap_{i\geq1}x_iR=(0),
	\]
	where each $x_iR$ is a nonzero prime ideal. Thus
	\[
	\bigcap_{\substack{P\in\operatorname{Spec}(R)\\P\neq(0)}}P=(0),
	\]
	so $R$ is not a $G$-domain. By
	Theorem~\ref{thm:compact-G-domain}, $GD(R)$ is not compact.
\end{example}

We conclude this section with the Noetherian property. Recall that a
\emph{quasi-order} is a reflexive and transitive relation. A
quasi-ordered set $(Q,\leq)$ is called \emph{well-quasi-ordered}
(or a \emph{wqo}) if every infinite sequence
\[
q_1,q_2,\ldots
\]
contains indices $i<j$ such that
\[
q_i\leq q_j.
\]
We refer to \cite[Definition~2.7 and Lemma~2.1]{BuriolaSchuster}
for this notion and its standard equivalent formulations. In
particular, $Q$ is a wqo if and only if every upward closed subset of
$Q$ is finitely generated; equivalently, every subset of $Q$ has only
finitely many minimal elements.

A classical characterization states that a quasi-ordered set is
well-quasi-ordered if and only if its Alexandroff topology is
Noetherian; see
\cite[Proposition~9.7.17]{GoubaultLarrecq}.

\begin{corollary}\label{cor:Noetherian-wqo}
	Let $R$ be a ring. Then $GD(R)$ is Noetherian if and only if
	$\mathcal R_{\mathrm{pr}}(R)$, ordered by inclusion, is
	well-quasi-ordered.
\end{corollary}

\begin{proof}
	By Theorem~\ref{thm:KQ},
	\[
	KQ(GD(R))\cong \mathcal R_{\mathrm{pr}}(R)
	\]
	with the upper Alexandrov topology. Since a topological space is
	Noetherian if and only if its Kolmogorov quotient is Noetherian, the
	result follows from
	\cite[Proposition~9.7.17]{GoubaultLarrecq}.
\end{proof}

\begin{remark}\label{rem:Noetherian-primes}
If $GD(R)$ is Noetherian, then
$\mathcal R_{\mathrm{pr}}(R)$ contains neither an infinite strictly
descending chain nor an infinite antichain. In particular, $R$ has no
infinite strictly descending chain of distinct principal prime ideals
and no infinite family of pairwise incomparable principal prime ideals.
\end{remark}

\begin{corollary}\label{cor:UFD-Noetherian}
	Let $R$ be a UFD. Then $GD(R)$ is Noetherian if and only if $R$ has
	only finitely many nonassociate prime elements.
\end{corollary}

\begin{proof}
	Suppose that $R$ has infinitely many nonassociate prime elements.
	The corresponding principal prime ideals form an infinite antichain
	in $\mathcal R_{\mathrm{pr}}(R)$. Hence
	$\mathcal R_{\mathrm{pr}}(R)$ is not well-quasi-ordered, and therefore
	$GD(R)$ is not Noetherian by
	Corollary~\ref{cor:Noetherian-wqo}.
	
	Conversely, suppose that, up to associates,
	\[
	p_1,\ldots,p_n
	\]
	are all the prime elements of $R$. For every $a\in R^\#$,
	\[
	a=u p_{i_1}^{e_1}\cdots p_{i_t}^{e_t}
	\]
	for some distinct $p_{i_1},\ldots,p_{i_t}$, and hence
	\[
	\sqrt{aR}
	=
	(p_{i_1}\cdots p_{i_t})R.
	\]
	Thus $\mathcal R_{\mathrm{pr}}(R)$ is finite, and therefore
	well-quasi-ordered. Hence $GD(R)$ is Noetherian by
	Corollary~\ref{cor:Noetherian-wqo}.
\end{proof}

\begin{corollary}\label{cor:UFD-finiteness}
	Let $R$ be a nonfield UFD. Then:
	\begin{enumerate}
		\item $GD(R)$ is Lindel\"of if and only if $R$ has at most
		countably many nonassociate prime elements.
		
		\item The following statements are equivalent:
		\begin{enumerate}
			\item $GD(R)$ is compact;
			\item $GD(R)$ is Noetherian;
			\item $R$ has only finitely many nonassociate prime elements.
		\end{enumerate}
	\end{enumerate}
\end{corollary}

\begin{proof}
	The first assertion is
	Corollary~\ref{cor:UFD-Lindelof}, while
	$(2)(b)\Leftrightarrow(2)(c)$ follows from
	Corollary~\ref{cor:UFD-Noetherian}.
	
	Suppose that $R$ has only finitely many nonassociate prime elements,
	say $p_1,\ldots,p_n$. Every nonzero prime ideal of $R$ contains a
	nonzero element, and hence contains some prime element $p_i$.
	Therefore
	\[
	0\neq p_1\cdots p_n
	\in
	\bigcap_{\substack{P\in\operatorname{Spec}(R)\\P\neq(0)}}P.
	\]
	Thus $R$ is a $G$-domain, and
	Theorem~\ref{thm:compact-G-domain} shows that $GD(R)$ is compact.
	
	Conversely, if $GD(R)$ is compact, then $R$ is a $G$-domain. Hence
	there exists
	\[
	0\neq a\in
	\bigcap_{\substack{P\in\operatorname{Spec}(R)\\P\neq(0)}}P.
	\]
	For every prime element $p$, the ideal $pR$ is a nonzero prime ideal,
	so $a\in pR$, and therefore $p\mid a$. Since a nonzero element of a
	UFD has only finitely many nonassociate prime divisors, $R$ has only
	finitely many nonassociate prime elements.
\end{proof}

\begin{remark}
	This gives another contrast with the ordinary divisor topology. For a
	nonfield integral domain, $D(R)$ is neither compact nor Noetherian;
	see \cite[Proposition~11 and Theorem~7]{YigitKoc}. In contrast,
	$GD(R)$ may satisfy both properties. For example, if
	\[
	R=k[[x]],
	\]
	then, up to associates, $R$ has a single prime element, and hence
	$GD(R)$ is both compact and Noetherian.
\end{remark}

\section{Multiplicative and quotient constructions}

The preceding sections show that the topology of $GD(R)$ is controlled
by radicals of principal ideals. We now examine how this structure
interacts with multiplication and with quotient maps. Two different
phenomena occur. Multiplication descends to $EC(R^\#)$ precisely in
the domain case, whereas quotient maps with nil kernel behave
naturally only after passing to the Kolmogorov quotient.

\subsection{Multiplication}

For rings with zero-divisors, the product of two nonzero nonunits may
vanish, so multiplication need not define a binary operation on
$EC(R^\#)$. In fact, for nonfield rings this is the only obstruction.

\begin{proposition}\label{prop:multiplication-domain}
	Let $R$ be a nonfield ring. Then the rule
	\[
	[a]\cdot[b]=[ab]
	\]
	defines a binary operation on $EC(R^\#)$ if and only if $R$ is an
	integral domain.
\end{proposition}

\begin{proof}
	Suppose first that $R$ is an integral domain. If $a,b\in R^\#$, then
	$ab$ is a nonzero nonunit, so $[ab]\in EC(R^\#)$. Moreover, if
	$aR=a'R$ and $bR=b'R$, then
	\[
	abR=(aR)(bR)=(a'R)(b'R)=a'b'R,
	\]
	so the operation is well defined.
	
	Conversely, suppose that the displayed rule defines a binary operation
	on $EC(R^\#)$. If $R$ were not a domain, there would exist nonzero
	$a,b\in R$ such that
	\[
	ab=0.
	\]
	Both $a$ and $b$ are nonunits, and hence $[a],[b]\in EC(R^\#)$.
	However, $[ab]=[0]$ does not belong to $EC(R^\#)$, a contradiction.
\end{proof}

Recall that a topological semigroup is a semigroup whose multiplication
is jointly continuous.

\begin{theorem}\label{thm:topological-semigroup}
	Let $R$ be a nonfield integral domain. Then $GD(R)$, endowed with the
	multiplication
	\[
	[a][b]=[ab],
	\]
	is a commutative topological semigroup.
\end{theorem}

\begin{proof}
	By Proposition~\ref{prop:multiplication-domain}, multiplication is
	well defined on $EC(R^\#)$. Consider
	\[
	\mu:GD(R)\times GD(R)\longrightarrow GD(R),
	\qquad
	\mu([a],[b])=[ab].
	\]
	Let $c\in R^\#$. By Lemma~\ref{lem:radical-divisibility},
	\[
	([a],[b])\in\mu^{-1}(B_c)
	\]
	if and only if
	\[
	\sqrt{cR}\subseteq\sqrt{abR}.
	\]
	Since
	\[
	\sqrt{abR}
	=
	\sqrt{aR}\cap\sqrt{bR},
	\]
	this is equivalent to
	\[
	\sqrt{cR}\subseteq\sqrt{aR}
	\quad\text{and}\quad
	\sqrt{cR}\subseteq\sqrt{bR}.
	\]
	Again by Lemma~\ref{lem:radical-divisibility}, this is equivalent to
	\[
	[a]\in B_c
	\quad\text{and}\quad
	[b]\in B_c.
	\]
	Consequently,
	\[
	\mu^{-1}(B_c)=B_c\times B_c.
	\]
	Thus the preimage of every basic open set is open, and hence
	$\mu$ is continuous.
\end{proof}

After passing to the Kolmogorov quotient, the multiplicative structure
takes a particularly simple order-theoretic form. Recall that a \emph{meet-semilattice} (called an \emph{inf
	semilattice} in \cite[Definition~O-1.8]{GierzEtAl}) is a partially
ordered set in which every two elements $x$ and $y$ have an infimum,
denoted by $x\wedge y$. A semilattice endowed with a topology is
called a \emph{topological semilattice} if its meet operation
\[
S\times S\longrightarrow S,
\qquad
(x,y)\longmapsto x\wedge y,
\]
is jointly continuous; see
\cite[Definition~VI-1.11]{GierzEtAl}.

\begin{corollary}\label{cor:radical-semilattice}
	Let $R$ be a nonfield integral domain, and let
	\[
	q:EC(R^\#)\longrightarrow KQ(GD(R))
	\]
	be the quotient map. The multiplication on $GD(R)$ descends to
	$KQ(GD(R))$, and under the identification
	\[
	KQ(GD(R))\cong\mathcal R_{\mathrm{pr}}(R),
	\qquad
	q([a])\longmapsto\sqrt{aR},
	\]
	the induced operation is the meet operation
	\[
	\sqrt{aR}\wedge\sqrt{bR}
	=
	\sqrt{aR}\cap\sqrt{bR}.
	\]
	Consequently, $KQ(GD(R))$ is a topological semilattice.
\end{corollary}

\begin{proof}
	Suppose that
	\[
	q([a])=q([a'])
	\quad\text{and}\quad
	q([b])=q([b']).
	\]
	By Theorem~\ref{thm:KQ},
	\[
	\sqrt{aR}=\sqrt{a'R}
	\quad\text{and}\quad
	\sqrt{bR}=\sqrt{b'R}.
	\]
	Hence
	\[
	\sqrt{abR}
	=
	\sqrt{aR}\cap\sqrt{bR}
	=
	\sqrt{a'R}\cap\sqrt{b'R}
	=
	\sqrt{a'b'R},
	\]
	and therefore
	\[
	q([ab])=q([a'b']).
	\]
	Thus multiplication descends to $KQ(GD(R))$.
	
	Under the identification
	\[
	KQ(GD(R))\cong\mathcal R_{\mathrm{pr}}(R),
	\qquad
	q([a])\longmapsto\sqrt{aR},
	\]
	we have
	\[
	\sqrt{abR}
	=
	\sqrt{aR}\cap\sqrt{bR}.
	\]
	Thus the induced multiplication corresponds to the meet operation on
	$\mathcal R_{\mathrm{pr}}(R)$. Since every meet-semilattice endowed
	with its upper Alexandrov topology is a topological semilattice, the
	last assertion follows.
\end{proof}

\subsection{Nil quotients}

We next consider surjective ring homomorphisms with nil kernel.
Given such a homomorphism $f:A\to B$, one might try to define
\[
EC(B^\#)\longrightarrow EC(A^\#),\qquad
[b]\longmapsto [x],
\]
where $x$ is a lift of $b$. This assignment, however, need not be
well defined, since different lifts of the same element of $B$ need
not be associates in $A$.

\begin{example}\label{ex:lifts-not-associate}
	Let
	\[
	A=k[t,\varepsilon]/(\varepsilon^2,t\varepsilon)
	\qquad\text{and}\qquad
	B=k[t].
	\]
	Every element of $A$ can be written uniquely in the form
	\[
	p(t)+\lambda\varepsilon,
	\qquad p(t)\in k[t],\ \lambda\in k.
	\]
	Consider the surjective homomorphism
	\[
	f:A\longrightarrow B,
	\qquad
	f\bigl(p(t)+\lambda\varepsilon\bigr)=p(t).
	\]
	Then
	\[
	\ker(f)=(\varepsilon)\subseteq\operatorname{Nil}(A),
	\]
	since $\varepsilon^2=0$.
	
	The element $t\in B^\#$ has two lifts $t$ and
	$t+\varepsilon$ in $A^\#$. These lifts are not associate. Indeed,
	\[
	tA=\{tp(t):p(t)\in k[t]\},
	\]
	and hence $t+\varepsilon\notin tA$. Therefore
	\[
	tA\neq(t+\varepsilon)A.
	\]
	Consequently, the assignment
	\[
	[b]\longmapsto[x],
	\]
	where $x$ is a lift of $b$, need not be well defined on associate
	classes.
	
	Nevertheless,
	\[
	\sqrt{tA}
	=
	\sqrt{(t+\varepsilon)A}
	=
	(t,\varepsilon).
	\]
	Indeed, $(t,\varepsilon)$ is a prime ideal of $A$ containing both
	$tA$ and $(t+\varepsilon)A$. Moreover, $\varepsilon^2=0$ gives
	$\varepsilon\in\sqrt{tA}$, while
	\[
	(t+\varepsilon)^2=t^2
	\]
	gives $t\in\sqrt{(t+\varepsilon)A}$ and hence also
	$\varepsilon\in\sqrt{(t+\varepsilon)A}$.
	Thus the ambiguity disappears after passing to the Kolmogorov
	quotient.
\end{example}

The preceding example suggests that the appropriate construction is
obtained after passing to the Kolmogorov quotient. By
Theorem~\ref{thm:KQ}, this amounts to working with
$\mathcal R_{\mathrm{pr}}(R)$ endowed with its upper Alexandrov
topology.

\begin{theorem}\label{thm:nilpotent-pullback}
	Let
	\[
	f:A\longrightarrow B
	\]
	be a surjective ring homomorphism such that
	\[
	\ker(f)\subseteq\operatorname{Nil}(A).
	\]
	Then taking preimages induces an order embedding
	\[
	\Phi_f:
	\mathcal R_{\mathrm{pr}}(B)
	\longrightarrow
	\mathcal R_{\mathrm{pr}}(A),
	\qquad
	J\longmapsto f^{-1}(J).
	\]
	Equivalently, under the identifications of
	Theorem~\ref{thm:KQ}, $f$ induces a topological embedding
	\[
	KQ(GD(B))
	\longrightarrow
	KQ(GD(A)).
	\]
	
	Moreover, the image of $\Phi_f$ contains every element of
	$\mathcal R_{\mathrm{pr}}(A)$ except possibly
	$\operatorname{Nil}(A)$. More precisely:
	\begin{enumerate}
		\item if $A$ is reduced or $B$ is not reduced, then $\Phi_f$ is surjective;
		\item if $A$ is not reduced and $B$ is reduced, then
		\[
		\operatorname{Im}(\Phi_f)
		=
		\mathcal R_{\mathrm{pr}}(A)
		\setminus\{\operatorname{Nil}(A)\}.
		\]
	\end{enumerate}
\end{theorem}

\begin{proof}
	Let
	\[
	J=\sqrt{bB}\in\mathcal R_{\mathrm{pr}}(B),
	\]
	where $b\in B^\#$, and choose $x\in A$ such that $f(x)=b$.
	Since the image of a unit is a unit, $x\in A^\#$. Put
	\[
	N=\ker(f).
	\]
	Then
	\[
	f^{-1}(bB)=xA+N,
	\]
	and hence
	\[
	f^{-1}(J)
	=
	f^{-1}(\sqrt{bB})
	=
	\sqrt{f^{-1}(bB)}
	=
	\sqrt{xA+N}.
	\]
	Since
	\[
	N\subseteq\operatorname{Nil}(A)\subseteq\sqrt{xA},
	\]
	we obtain
	\[
	f^{-1}(J)=\sqrt{xA}.
	\]
	Thus $\Phi_f$ is well defined.
	
	Preimages preserve inclusion. Since $f$ is surjective,
	\[
	f(f^{-1}(J))=J
	\]
	for every ideal $J$ of $B$. It follows that $\Phi_f$ is injective
	and reflects inclusion, and hence is an order embedding. Under the
	upper Alexandrov topologies, it is therefore a topological
	embedding.
	
	We now determine its image. Let
	\[
	I=\sqrt{aA}\in\mathcal R_{\mathrm{pr}}(A).
	\]
	If
	\[
	I\neq\operatorname{Nil}(A),
	\]
	then $a$ is not nilpotent. Since $N\subseteq\operatorname{Nil}(A)$,
	we have $a\notin N$, and hence $f(a)\neq0$.
	
	Moreover, $f(a)$ is a nonunit. Indeed, if $f(a)$ were a unit,
	surjectivity of $f$ would give $c\in A$ such that
	\[
	f(ac)=1.
	\]
	Thus
	\[
	1-ac\in N\subseteq\operatorname{Nil}(A),
	\]
	so $ac$ is a unit, forcing $a$ to be a unit, a contradiction.
	Hence $f(a)\in B^\#$, and
	\[
	\Phi_f\bigl(\sqrt{f(a)B}\bigr)
	=
	\sqrt{aA}
	=
	I.
	\]
	Therefore the only possible element of
	$\mathcal R_{\mathrm{pr}}(A)$ not in the image is
	$\operatorname{Nil}(A)$.
	
	We also have
	\[
	f^{-1}(\operatorname{Nil}(B))
	=
	\operatorname{Nil}(A).
	\]
	Indeed, the inclusion from right to left is immediate. Conversely,
	if $f(a)$ is nilpotent, then $a^n\in N$ for some $n$, and since
	$N\subseteq\operatorname{Nil}(A)$, it follows that $a$ is
	nilpotent.
	
	If $B$ is not reduced, choose a nonzero nilpotent $b\in B$.
	Then $b\in B^\#$ and
	\[
	\sqrt{bB}=\operatorname{Nil}(B),
	\]
	so
	\[
	\Phi_f(\operatorname{Nil}(B))
	=
	\operatorname{Nil}(A).
	\]
	Hence $\Phi_f$ is surjective.
	
	If $A$ is reduced, then $N=0$, so $f$ is an isomorphism and
	$\Phi_f$ is again surjective.
	
Finally, suppose that $A$ is not reduced and $B$ is reduced. Then
$\operatorname{Nil}(A)\in\mathcal R_{\mathrm{pr}}(A)$. If
$\operatorname{Nil}(A)=\Phi_f(J)$ for some
$J\in\mathcal R_{\mathrm{pr}}(B)$, then
\[
J=f\bigl(f^{-1}(J)\bigr)
=f(\operatorname{Nil}(A))
=\operatorname{Nil}(B)
=(0),
\]
a contradiction. Therefore
\[
\operatorname{Im}(\Phi_f)
=
\mathcal R_{\mathrm{pr}}(A)
\setminus\{\operatorname{Nil}(A)\}.
\]
\end{proof}

\begin{remark}
	The map in Theorem~\ref{thm:nilpotent-pullback} is naturally
	described in terms of preimages of radical principal ideals:
	\[
	J\longmapsto f^{-1}(J).
	\]
	In particular, although lifts are used in the proof, the resulting
	map is canonical and involves no choice of lifts.
\end{remark}

The preceding result takes a particularly simple form for quotients by
nil ideals.

\begin{corollary}\label{cor:nil-quotient}
	Let $A$ be a ring and let
	\[
	I\subseteq\operatorname{Nil}(A)
	\]
	be an ideal. Then the quotient map
	\[
	\pi:A\longrightarrow A/I
	\]
	induces a topological embedding
	\[
	KQ(GD(A/I))
	\longrightarrow
	KQ(GD(A)).
	\]
	More precisely:
	\begin{enumerate}
		\item if
		\[
		I\subsetneq\operatorname{Nil}(A),
		\]
		then
		\[
		KQ(GD(A/I))
		\cong
		KQ(GD(A));
		\]
		
		\item if $A$ is not reduced and
		\[
		I=\operatorname{Nil}(A),
		\]
		then, under the identification of
		Theorem~\ref{thm:KQ}, the image of
		\[
		KQ(GD(A/\operatorname{Nil}(A)))
		\longrightarrow
		KQ(GD(A))
		\]
		corresponds to
		\[
		\mathcal R_{\mathrm{pr}}(A)
		\setminus\{\operatorname{Nil}(A)\}.
		\]
	\end{enumerate}
\end{corollary}

\begin{proof}
	Apply Theorem~\ref{thm:nilpotent-pullback} to the quotient map.
	If
	\[
	I\subsetneq\operatorname{Nil}(A),
	\]
	then $A/I$ contains a nonzero nilpotent element and is therefore
	not reduced. Hence the induced embedding is surjective.
	
	If
	\[
	I=\operatorname{Nil}(A)
	\]
	and $A$ is not reduced, then $A/I$ is reduced, and the second
	assertion follows from Theorem~\ref{thm:nilpotent-pullback}.
\end{proof}

\begin{remark}
	Thus reduction modulo the nilradical has a particularly simple
	effect on the Kolmogorov quotient. If $A$ is not reduced, then,
	under the identification
	\[
	KQ(GD(A))\cong\mathcal R_{\mathrm{pr}}(A),
	\]
	the space
	\[
	KQ(GD(A_{\mathrm{red}})),
	\qquad
	A_{\mathrm{red}}=A/\operatorname{Nil}(A),
	\]
	corresponds precisely to
	\[
	\mathcal R_{\mathrm{pr}}(A)
	\setminus\{\operatorname{Nil}(A)\}.
	\]
	In other words, reduction removes the least element
	$\operatorname{Nil}(A)$ and leaves the remaining
	radical-principal order unchanged.
\end{remark}

\end{document}